 \documentclass[12pt]{amsart}

\usepackage{amssymb, amsmath, amsfonts}
\usepackage[raiselinks=false,colorlinks=true,citecolor=blue,urlcolor=blue,
linkcolor=blue,bookmarksopen=true,pdftex]{hyperref}
 
\usepackage{xcolor}
\usepackage{enumerate}
\usepackage[mathscr]{eucal}
\usepackage{mathrsfs}

\newtheorem{theorem}{Theorem}[section]
\newtheorem{proposition}[theorem]{Proposition}
\newtheorem{lemma}[theorem]{Lemma}
\newtheorem{remark}[theorem]{Remark}

\newtheorem{example}[theorem]{Example}
\newtheorem{corollary}[theorem]{Corollary}

\newcommand{\Aut}{\mathrm{Aut}}

\newcommand{\tr}{\mathrm{tr}}

\newcommand{\fix}{\mathrm{Fix}}

\newcommand{\SL}{\mathrm{SL}}
\newcommand{\GL}{\mathrm{GL}}
\newcommand{\PGL}{\mathrm{PGL}}
\begin{document}
\baselineskip=15.5pt

\title[Twisted conjugacy in general linear groups]{Twisted conjugacy in $\SL_n$ and $\GL_n$ over subrings of $\bar{\mathbb F}_p(t)$.} 
\author[O.  Mitra] {Oorna Mitra } 
\email{oornam@cmi.ac.in} 
\author[P. Sankaran]{ Parameswaran Sankaran}
\email{sankaran@cmi.ac.in} 
\address{Chennai Mathematical Institute, SIPCOT IT Park, Siruseri, Kelambakkam, 603103.}
\thispagestyle{empty}
\date{}

\begin{abstract}
Let $\phi:G\to G$ be an automorphism of an infinite group $G$.  One has an equivalence relation $\sim_\phi$ on $G$ defined as 
$x\sim_\phi y$ if there exists a $z\in G$ such that $y=zx\phi(z^{-1})$.  The equivalence classes are called $\phi$-twisted conjugacy classes 
and the set $G/\!\!\sim_\phi$ of equivalence classes is denoted $\mathcal R(\phi)$. The cardinality $R(\phi)$ of $\mathcal R(\phi)$ is 
called the Reidemeister number of $\phi$.  We write $R(\phi)=\infty$ when $\mathcal R(\phi)$ is infinite.  We say that $G$ has the $R_\infty$-{\it property} if  $R(\phi)=\infty$ for every automorphism $\phi$ of $G$.    We show that the groups $G=\GL_n(R), \SL_n(R)$  have the $R_\infty$-property for all $n\ge 3$ 
when $ F[t]\subset R\subsetneq  F(t)$ where $F$ is a subfield of $\bar{\mathbb F}_p$.  When $n\ge 4$, we show that any subgroup $H\subset \GL_n(R)$ that contains $\SL_n(R)$ also has the $R_\infty$-property.  
\end{abstract}

\maketitle

\section{Introduction} 
Given an endomorphism $\phi:G\to G$ of a group $G$, one has the $\phi$-twisted conjugacy action of $G$ on itself 
defined by $g.x=gx\phi(g^{-1})$.  The orbits of this action are called the $\phi$-twisted conjugacy classes.  
The cardinality of the orbit space $\mathcal R(\phi)$ is called the Reidemeister number of $\phi$ and is denoted $R(\phi)$.  
When the orbit space is infinite 
we write $R(\phi)=\infty$.  One says that $G$ has the $R_\infty$-property if $R(\phi)=\infty$ for every automorphism 
$\phi$ of $G$.

The notion of Reidemeister number originated in Nielsen fixed point theory.  
The problem of determining which classes of countably infinite groups have the $R_\infty$-property was 
 formally posed by Fel'shtyn and Hill \cite{felshtyn-hill} and is an 
active area of research in group theory.    Many classes of groups have been classified according to whether or not they have the $R_\infty$-property.   This is an interesting classification problem because there is no 
specific method or procedure to decide whether a given group or class of groups have the $R_\infty$-property.   
The tools needed are to be developed depending on the class of groups under consideration and have often involved using results and techniques drawn from other branches of mathematics, including 
geometric group theory \cite{ll}, combinatorial group theory \cite{g-w}, Bass-Serre theory \cite{mitra-sankaran-n2},  number theory \cite[\S 6]{gss}, algebraic groups and Lie theory \cite{ms-tg}, etc.

 In this paper we address this question for the 
class of general and special linear groups, $\GL_n(R), \SL_n(R), n\ge 3$, over an integral domain $R$ such that 
$F[t]\subset R\subsetneq F(t)$ where $F$ is a subfield of $\bar{\mathbb F}_p$, $p,$ a prime.

Let $K=\mathbb{F}_q(t)$ and $S$ be a non-empty subset of the set $\mathscr{V}$  of all inequivalent valuations on $K$. 
Then $K(S) := \{x \in K \mid \nu(x) \geq 0~\forall \nu \notin S\}$ is called the ring of $S$-integers of $K$. 
Note that if $R$ is a ring such that $\mathbb{F}_q[t] \subset R \subsetneq K$, then $R$ is a ring of $S$-integers of K, for an appropriate non-empty, possibly infinite subset $S$. In fact, for any non-empty, proper subset $S$ of $\mathscr{V}$, we have $\mathbb{F}_q[u] \subset K(S) \subsetneq \mathbb{F}_q(u) \subset K$ for an element $u\in K$ that is transcendental over $\mathbb F_q$. So without loss of generality, we may assume $\nu_\infty \in S$ where $\nu_\infty(f(t)/g(t))= \deg(g(t))-\deg(f(t))$, in which case $\mathbb{F}_q[t] \subset K(S) \subsetneq K$ holds.

Now, for all $n \geq 2$, $\SL_n(K(S))$ is an $S$-arithmetic subgroup of $\SL_n(K)$ and if $S$ is finite, it is a lattice in $\prod_{\nu \in S}\SL_n(K_\nu)$ when embedded diagonally, where $K_\nu$ is the completion of $K$ with respect to  the valuation $\nu$ \cite[(3.2.5), p. 63]{mar}. It was shown in \cite{ms-tg} that any irreducible lattice in a connected non-compact semisimple real Lie group with finite centre has the $R_\infty$-property. The present work may be viewed as a first step in classifying, according to the $R_\infty$-property, lattices in semisimple connected algebraic groups over local fields of positive characteristics.  
We consider only the case $n\ge 3$ here.  The groups $\GL_2(R), \SL_2(R)$ required an entirely different approach.  It has been shown that 
that $\GL_2(\mathbb F_q[t]), \GL_2(\mathbb F_q[t,t^{-1}]), \SL_2(\mathbb F_q[t])$ have the 
$R_\infty$-property.  See  \cite{mitra-sankaran-n2}.

The $R_\infty$-property for linear 
groups over fields of characteristic zero was considered by Nasybullov \cite{nasyb1}, \cite{nasyb2} and 
Felshtyn and Nasybullov \cite{fn} culminating in the result that 
a Chevally group of classical type over an algebraically closed 
field $F$ of characteristic zero has the $R_\infty$-property if and only if 
$F$ has finite transcendence degree over $\mathbb Q$.
 
We now state our main results.

\begin{theorem}  \label{fpbar} Let $n\ge 3$ and let 
 $F$ be a subfield of  $\bar{\mathbb F}_p$. 
Let $R$ be an integral domain such that $F[t]\subset R\subsetneq F(t)$.  
Then the groups $G=G(R)=\GL_n(R), \SL_n(R)$ have the $R_\infty$-property. 
\end{theorem}

\begin{theorem}\label{sandwich}  Let $R$ be as in Theorem \ref{fpbar}. 
Suppose that $H$ is any group such that $\SL_n(R)\subset H\subset \GL_n(R)$.   Then $H$ 
has the $R_\infty$-property in the following cases: (a) $n\ge 4$, (b) $n=3$ and $\det(H)\subset F^\times$.  

\end{theorem}

We have not been able to show the $R_\infty$-property for $H$ as in Theorem \ref{sandwich} unconditionaly when $n=3$.

Our proof of Theorem \ref{fpbar} relies heavily on some classical results concerning general and special linear groups over integral domains which are not fields, namely: (i) 
the description, due to O. T. O'Meara \cite{omeara}, of the automorphisms of $G=\SL_n(R), \GL_n(R), n\ge 3$,
which is valid for any integral 
domain $R$ that is not a field,  (ii)  the group $\SL_n(R), n\ge 3,$ is perfect and is generated by elementary matrices 
when $R$ is a Euclidean domain, (see \cite{rosenberg}).  

We shall prove Theorem \ref{fpbar} first for the case when $F$ is a finite field $\mathbb F_q$.  In this case, we prove that 
every outer automorphism $[\phi]$ of $G=\GL_n(R)$ or $\SL_n(R)$ is represented by a convenient automorphism $\phi$ using O'Meara's theorem.   The proof in this case crucially uses the observation that when $\phi$ is induced by 
an automorphism of the ring $R$, it  
 has finite order.   
We also show that, in this case, the fixed subgroup of $\phi$ contains a 
subgroup isomorphic to $\SL_n(\mathbb F_p[t])$.    In the more general case, it is not true that 
the outer automorphisms are represented by finite order automorphisms.  
(For example, automorphisms of $G$ induced by a Frobenius automorphism of the ring $R$.)   Nevertheless,  any automorphism that is induced by an automorphism of the ring $\phi:R\to R$, 
restricts to an automorphism of $G_q= G\cap \GL_n(R_q)$ where $R_q=R\cap 
\mathbb F_q(t)$ for some $q=p^e$.   This allows us to extend 
our proof technique in case when $F$ is a finite field to the more general case when $F\subset \bar {\mathbb F}_q$.

Although the automorphism group of $H$ as in Theorem \ref{sandwich} is not known to us, any automorphism $\theta$ of 
$H$ restricts to an automorphism $\theta'$ of $\SL_n(R)$.  After choosing a convenient representative of the 
outer automorphism class of $\theta$, we find a sequence $(y_k)$ of elements of $\SL_n(R)$ which are in pairwise 
distinct $\theta'$-twisted conjugacy classes and show that they remain in pairwise distinct $\theta$-twisted conjugacy classes, under the hypotheses of the theorem. 

Recently, Garge and Mitra \cite{gm} have obtained  results analogous to those of 
Theorem \ref{fpbar} for the classical groups $\mathrm{SO}_n(R), \mathrm{Sp}_n(R)$.

E. Jabara \cite[Theorem C]{jabara} has shown that a finitely generated linear group that admits 
a finite order automorphism $\phi$ with finite Reidemeister number is necessarily virtually solvable.  His proof makes 
use of the deep work of Hrushovski, Kropholler, Lubotzky, and Shalev \cite{hkls}.   It should be noted that 
typically the groups $\SL_n(R), \GL_n(R)$ with $R$ as in our main Theorem are not finitely generated.


\section{Some basic results on twisted conjugacy}\label{basicresults}
Let $\phi:G\to G$ be an automorphism of a group $G$.  We shall denote by $[x]_\phi$ the $\phi$-twisted conjugacy class of $x$ and 
by $\mathcal R(\phi)$ the set of all $\phi$-twisted conjugacy classes in $G$.  

We collect here some basic results concerning twisted conjugacy and the $R_\infty$-property which are relevant for our purposes. 
Let $G$ be an infinite group (not necessarily finitely generated) and let $K\subset G$ be a normal subgroup.  Let $\eta:G\to H$ be the 
canonical quotient map where $H=G/K$.   Suppose that $\phi:G\to G$ is an automorphism such that $\phi(K)=K$ so that we have 
the following diagram in which the rows are exact and isomorphisms $\phi',\bar\phi$ are defined by $\phi$:
\[ 
\begin{array}{ l l l l l l llll }
1 &\to & K&\to & G& \to & H& \to &1\\
& &\downarrow\phi' && \downarrow \phi& &\downarrow \bar \phi &&\\
1&\to &K&\to &G& \to &H&\to& 1\\
\end{array}\eqno(1)
\]
Any $\phi$-twisted conjugacy class in $G$ maps into a $\bar\phi$-twisted conjugacy class in $H$.  Also, any $\phi'$-twisted conjugacy 
class in $K$ is contained in a $\phi$-twisted conjugacy class of $G$.   Moreover, if $H$ is finite, then any $\phi$-twisted conjugacy 
class contains at most $o(H)$ many distinct $\phi'$-twisted conjugacy classes of $K$.  So, if $R(\phi')=\infty$, then $R(\phi)=\infty$.  
See \cite[Lemma 2.2]{ms-cmb} 
and its proof.   We summarise these results as a lemma.

We recall that a subgroup $K$ of $G$ is {\it characteristic} in $G$ if every automorphism of $G$ restricts to an automorphism of $K$.     

\begin{lemma} \label{characteristic}
Suppose that $\phi:G\to G$ is an automorphism of an infinite group such that the rows in (1) are exact and 
the homomorphisms $\phi',\bar \phi$ are isomorphisms. Then: \\
(i) If $R(\bar\phi)=\infty$, then $R(\phi)=\infty$.  In particular, if $K$ is characteristic in $G$, then $G$ has 
the $R_\infty$-property if $H$ does.\\
(ii) Suppose that $H$ is finite. Then $R(\phi)=\infty$ if $R(\phi')=\infty$.  In particular, if $K$ is characteristic and 
has finite index in $G$, then $G$ has the $R_\infty$-property if $K$ does. \hfill $\Box$ \\
\end{lemma}

The following lemma is well-known and can be found, for example, in \cite[\S3]{g-s}.

Let $g\in G$ and let $\iota_g:G\to G$ denote the inner automorphism $x\mapsto gxg^{-1}$.  

\begin{lemma}\label{outer} 
If $\phi:G\to G$ is any automorphism, then $[x]_{\iota_g\circ\phi}\mapsto [xg]_{\phi}$ defines 
a bijection $\mathcal R(\iota_g\circ \phi)\to \mathcal R(\phi)$.  In particular,
$R(\iota_g\circ \phi)=R(\phi)$.   \hfill $\Box$
\end{lemma}

The following lemma can be proved along the same lines as \cite[Lemma 3.4]{g-s} or \cite[ Lemma 2.3]{gs-pjm-2016}.  
For the sake of completeness we give the proof.
\begin{lemma} \label{fixedgroup}
Let $\theta:G\to G$ be a finite order automorphism.  Let $r=o(\theta).$
(i)  Suppose that $[x]_\theta=[y]_\theta$.  Then $\prod_{0\le j<r}\theta^j(x)$ and $\prod_{0\le j<r}\theta^j(y)$ are conjugates in $G$.
(ii) Suppose that $x,y \in \fix(\theta)$ and that 
$[x]_\theta=[y]_\theta$.  Then $x^r$ and $y^r$ are conjugates in $G$. \\
\end{lemma}
\begin{proof}
Suppose that $y=zx\theta(z^{-1} )$ for some $z\in G$. Applying $\theta^j$ to both sides of the equality we obtain 
that $\theta^j(y)=\theta^j(z) \theta^j(x)\theta^{j+1}(z^{-1})$.   Now taking product in order as we vary $j$ in $\{0,1,\ldots, r-1\}$, we obtain that $\prod_{0\le j<r} \theta^j(y)=\prod_{0\le j<r} \theta^j (z) \theta ^j(x)\theta^{j+1}(z^{-1})
=z\prod_{0\le j<r}\theta^j(x) \theta^r(z^{-1})$.  Since $\theta^r=id$, the first assertion follows. The second 
assertion is a special case of the first. 
\end{proof}


\section{Automorphisms of $\GL_n(R)$ and $\SL_n(R)$} \label{omeara}
In this section we recall some properties of the groups $\SL_n(R)$ and $\GL_n(R)$ as well as their 
automorphisms where $R$ is an integral domain, but not a field with $F[t] \subset R \subset F(t)$ and $F$ is a field. Up to the end of \S\ref{commutations},  we do not assume that $F$ is a subfield of $\bar{\mathbb F}_p$.  

Since $F[t]$ is a PID, any such $R$ is a localization of $F[t]$ by some multiplicatively closed subset of $F[t]$. Moreover, such an $R$ is a Euclidean domain as any localization of a Euclidean domain is again a Euclidean domain.

Let $E_n(R)\subset \GL_n(R)$ denote the subgroup generated by 
the elementary matrices $e_{ij}(\lambda),\lambda \in R, 1\le i, j\le n, i\ne j$.  
By definition, $e_{ij}(\lambda)$ is the matrix 
whose diagonal entries are $1$, the $(i,j)$-th entry is $\lambda$ and all other entries are zero.  

One has the obvious inclusions 
$E_n(R)\subset \SL_n(R)\subset \GL_n(R)$.    As $R$ is a Euclidean domain, we have $E_n(R)=\SL_n(R)$ 
for $n\ge 2$; see \cite[Theorem 2.3.2]{rosenberg}.   When $n\ge 3$, we also have $E_n(R)=[E_n(R),E_n(R)]$; this follows 
from the observation that the commutator $[e_{ik}(x),e_{kj}(1)]$ equals $e_{ij}(x)~\forall x\in R$ if $i,j,k$ are all distinct.    It follows 
that, when $n\ge 3$,  $\SL_n(R)$ is perfect and it equals the derived group $[H,H]$
where $H$ is any subgroup of $\GL_n(R)$ that contains $\SL_n(R)$.  
 
\noindent
{\bf Notations:} {We denote by $\delta(a_1,\ldots, a_n)\in GL_n(R)$ the diagonal matrix with $(i,i)$-entry equals $a_i\in R^\times, 1\le i\le n$.  Also, we denote by $\delta(A_1,\ldots, A_k)$ the block diagonal matrix with $j^{th}$ diagonal block is equal 
to $A_j$.  We denote by $h(a)$ the matrix $\delta(I_{n-1},a)\in \GL_n(R)$ where $a\in R^\times$.  As above, for $\lambda\in R $ and $i\ne j$, $e_{ij}(\lambda)\in \SL_n(R)$ is 
the unipotentent triangular matrix whose $(i,j)$-th entry is $\lambda$ and all other 
non-diagonal entries are zero.}

\subsection{Theorem of O'Meara}
Let $n\ge 3$ and let $G$ be one of the groups $\SL_n(R),  \GL_n(R)$ where $R$ is any (commutative) integral domain which is not a field. Let $K$ be the fraction field of $R$.
An automorphism $\phi:G\to G$ is called {\it standard} if it is in the subgroup generated by the following four types of automorphisms:\\
(i) {\it Conjugation}:  Conjugation by $g\in \GL_n(K)$, denoted $\iota_g:G\to G,$ is defined as  $x\mapsto gxg^{-1}$.    
It is inner if $g\in G$.\\
(ii) {\it Ring automorphism}: These are automorphisms of $G$ induced by automorphisms of the ring $R$.  
We make no distinction in the notation between the automorphism of the ring 
$R$ and the induced automorphism of $G$.  \\
(iii) {\it Homothety}:   Recall that a homomoprhism $\mu=\mu_\chi: G\to G$ is a {\it homothety} 
if there is a character $\chi:G\to R^\times$ 
such that $\mu_\chi(x)=\chi(x)x$.  Since $x,\chi(x)x\in G$, it follows that $\chi(x)I_n\in G$. Being a scalar matrix, $\chi(x)I_n$ belongs to the centre of $G$.  
A homothety $\mu_\chi$ fails to to be injective if and only if there exists a central element $zI_n\in G$ other than $I_n$ such that $\chi(zI_n)=z^{-1}$.  

Let $\mu_\chi$ be an automorphism of $G$.    
Observe that,  since $\SL_n(R)=[G,G]$ is characteristic in $G$, $\mu_\chi$ restricts to the identity automorphism on $\SL_n(R)$.
In particular, any homothety is the identity automorphism when $G=\SL_n(R)$.  \\
(iv) {\it Contragredient}:   The contragredient automorphism $\epsilon:G\to G$, is defined as $x\mapsto {}^{t}\!x^{-1},~\forall x\in G$.

O. T. O'Meara \cite{omeara} has shown that for any integral domain $R$ which is not a field, any automorphism of $G$ is standard, provided $n\ge 3$. 

\begin{theorem} {\em (O'Meara \cite{omeara})} \label{aut-omeara}
Let $R$ be an integral domain 
with fraction field $K\ne R$ and 
let $G=\GL_n(R)$ or $\SL_n(R)$ where $n\ge 3$.  
Then any automorphism $\phi:G\to G$ can be expressed as follows:
\[\phi=\mu_\chi\circ \rho\circ \iota_g ~~\textrm{or}~~\phi=\mu_\chi\circ \rho\circ \iota_g \circ \epsilon\]
where $\mu_\chi$ is a homothety automorphism corresponding to a character $\chi: G\to R^\times, g\in \GL_n(K), \rho:G\to G$ is induced by a ring 
automorphism denoted by the same symbol $\rho:R\to R$ and $\epsilon$ is the contragredient $x\mapsto {}^t\! x^{-1}$.
\end{theorem}

\begin{remark} \label{remarkonomeara}
{\em 
O'Meara's theorem is stated in the module theoretic language.  His expression for an automorphism involves 
a semilinear automorphism $\Phi_g$ of $M= R^n$, which  
corresponds to the composition $\rho \circ \iota_g$ where 
$\rho:K\to K$ is a field automorphism and a $K$-linear isomorphism $g:K^n\to K^n$. 
 His Theorem B  in \S5 (op. cit.) states that a seimiliear automorphism of $K^n$ yields an automorphism of $M$ if 
and only if $g(M)=\mathfrak a.M$ for an invertible fractional ideal $\mathfrak a$ of $R$ and $\rho(R)=R$.  When $R$ is a PID (as in our case), every fractional ideal is invertible and moreover $\mathfrak a=\lambda R$ for some $\lambda\in K^\times$.  

Theorem C in \S5, {\it loc cit.}, 
is the analoguous statement when the contragredient is involved. Thus, when $R$ is a subring of $\bar{\mathbb F}_p(t)$
which is not a field, the element $g$ in Theorem \ref{aut-omeara} may be taken to be in $\GL_n(R)$.  }
\end{remark}

\subsection{Commutation relations} \label{commutations}
We have the following commutation relations:\\ (i) $ \rho\circ \iota_g=\iota_{\rho(g)}\circ \rho, \mu_\chi\circ \iota_g=\iota_g\circ \mu_\chi,$\\ 
(ii) $\epsilon\circ \rho=\rho\circ \epsilon,\\ $
(iii) $\epsilon \circ \iota_g=\iota_{\epsilon(g)}\circ \epsilon$, \\ 
(iv) $\mu_\chi\circ \rho= \rho\circ \mu_\eta$ where $\eta=\rho_0^{-1}\circ \chi\circ \rho:G\to R^\times$. Here $\rho_0:R^\times\to R^\times$ is defined by $\rho:R\to R$.\\ 
(v) $\mu_{-\chi}\circ \epsilon=\epsilon\circ \mu_{\chi\circ \epsilon}$ where $-\chi(g)=\chi(g)^{-1}~\forall g\in G$.  

\begin{example}\label{thetasquare}
Let $\theta_j=\iota_{g_j}\circ \rho_j\circ \epsilon,$ where $g_j\in \GL_n(R), j=1,2.$  Then 
\[\begin{array}{rcl}
\theta_1\circ \theta_2&=&\iota_{g_1} \circ \rho_1\circ \epsilon\circ \iota_{g_2}\circ \rho_2\circ \epsilon\\
&=&\iota_{g_1}\circ \rho_1\circ\iota_{^tg^{-1}_2}\circ \epsilon\circ  \rho_2\circ \epsilon \\
&=& \iota_{g_1}\circ \iota_{\rho_1(^tg_2^{-1})}\circ \rho_1\circ \rho_2\\
&=&\iota_{g}\circ \rho,\\
\end{array}\]
where $g=g_1\rho_1(^tg_2^{-1})$ and  $\rho=\rho_1\circ \rho_2.$
\end{example}

\subsection{Ring automorphisms} \label{ring-auto} 
Let $R$ be a ring such that $F[t]\subset R\subset K:=F(t)$ where $F$ is a subfield of $\bar {\mathbb F}_p$.   Any ring  automorphism $\rho:R\to R$ extends to an automorphism $\tilde \rho$ of $K$.   
Since $\tilde \rho$ preserves the set of all elements of finite order in $K^\times,$ it follows that $ \rho$ restricts to an automorphism of $F$.  
Since $F\subset \bar{ \mathbb F}_p$, $\rho|_{F}$ is a {\it Frobenius automorphism } $x\mapsto x^{p^r}$ (the value of $r$ may depend on $x\in F$).   Note that any automorphism $\phi$ of $F$ extends to a unique automorphism $\tilde \phi$ of $K$ that fixes $t$.    Evidently $\tilde \phi$ restricts to an automorphism of $R,$ again denoted $\tilde \phi$.   
We shall refer to $\tilde \phi\in \Aut(R)$ as a {\it Frobenius automorphism}.
If $\rho:R\to R$ is any 
automorphism then $\psi^{-1}\circ \rho=:\rho_1$ is an $F$-automorphism of $R$ where $\psi$ is the Frobenius 
automorphism of $R$ defined by $\rho|_F\in Aut(F)$.  

Let $\gamma=\left (\begin{smallmatrix}a&b\\c&d\end{smallmatrix}\right )\in GL(2,F).$   Let  $ M_\gamma\in Aut_F(K)$ be 
the {\it M\"obius automorphism} defined as $M_\gamma(t)=\frac{at+b}{ct+d}$.  It is convenient to write $\gamma$ to also 
denote the M\"obius automorphism $M_\gamma$.  Clearly $M_\gamma=M_{\gamma'}$ if and only if 
$\gamma'=\lambda.\gamma$ for some $\lambda\in F^\times$. 
Thus the group of all 
M\" obius automorphisms of $K$ is isomorphic to the group $\mathrm{PGL}(2,F)$.  When $F[t]\subset R\subset K$ and $\gamma$ 
is a M\"obius automorphism of $K$ that stabilizes $R$, we say that $\gamma|_R$ is a M\"obius automorphism and again denote this by $\gamma$.   A M\"obius automorphism is also referred to as a {\it  M\"obius transformation.}  

Let  $\phi:F\to F$ be an automorphism.  Since $F\subset \bar{\mathbb F}_p$, $\phi$ is a Frobenius automorphism.   
We have the induced  automorphism of $\phi$ on $\mathrm{PGL}(2,F)$. Thus we 
obtain an action of $\Phi=\Aut(F)$ on $\mathrm{PGL}(2,F)$.   We regard $\phi$ as an automorphism $\phi:R\to R$
where $\phi(t)=t$, and also when $R=K$.  We have the commutation relation \[\phi\circ \gamma=\phi(\gamma)\circ \phi.\]
  To see this, 
we need only note that, writing $\gamma=\bigl(\begin{smallmatrix} a&b\\c&d\end{smallmatrix}\bigr)$, 
$\phi\circ \gamma(t)=\phi((at+b)/(ct+d))=(\phi(a)t+\phi(b))/(\phi(c)t+\phi(d))=\phi(\gamma)(\phi(t))$ and $\phi\circ \gamma(\lambda)=\phi(\lambda)~\forall \lambda\in F$. 

We shall denote by $\Aut_F(R)$ the group of $F$-automorphisms of $R$.  
\begin{lemma} \label{autK}
(i) Any $F$-automorphism of $K=F(t)$ is a M\"obius transformation $\gamma:K\to K$. 
The group $\Aut(K)$ is a semidirect product $\Aut_F(K)\rtimes \Phi=\mathrm{PGL}(2,F)\rtimes \Phi$ where $\Phi$ is the group of all automorphisms of $F$.  \\
(ii) The group $\Aut(R)$ is a semidirect product $\Aut_F(R)\rtimes \Phi\subset \Aut(K)$.  In particular, 
any automorphism $\rho:R\to R$ can be expressed uniquely 
as $\rho=\rho_0\circ \varphi$ where $\varphi$ is a Frobenius automorphism and $\rho_0$ is a M\"obius automorphism of $R$. 
\end{lemma}
\begin{proof}  (i) 
By the discussion preceding the statement of the lemma, we have a split exact sequence 
\[ 1\to Aut_F(K)\to Aut(K)\to \Phi\to 1\]
where $Aut(K)\to \Phi$ is the restriction to $F$.  The splitting is given by $\varphi\mapsto \tilde \varphi$ (in the above notation) which is indeed an automorphism.  So to complete the proof it suffices to show that 
any $F$-automorphism $\sigma:K\to K$ is a M\"obius automorphism.   Let $\tau:K\to K$ be the inverse of $\sigma$.

Let $\sigma(t)=f(t)/g(t)$ and let  $\tau(t)=h(t)/k(t)$, where $f(t),g(t),h(t),k(t)\in F[t]$ where $\gcd(f(t),g(t) )=1=\gcd(h(t),k(t))$.  We need to show that 
$\deg(f), \deg(g) <2$.  We assume that $\deg(f)\ne deg(g)$; otherwise we may replace $\sigma$ by the composition 
the automorphism $t\mapsto \sigma(t)-\alpha$ where we choose $\alpha$ such that 
$f(t)-\alpha g(t)$ has degree strictly less than $\deg(g)$.  Since $t\mapsto 1/t$ defines an automorphism we may assume without loss of generality that $m:=\deg(f)>\deg(g)=:n.$ 
Suppose $n=0$.   Then $g(t)\in F$ and $\sigma(t)\in F[t]$ is a polynomial of degree at least $1$. If $\deg(f(t))=1$, then 
$\sigma$ ia a M\"obius transformation.  If $\deg(f(t))\ge 2$, then it is readily seen that $t$ is not in the image of $
\sigma$.  Now suppose that $n=\deg(g(t))\ge 1$.

Let $f(t)=\sum_{0\le j\le m} f_jt^j, g(t)=\sum_{0\le i\le n}g_it^i$ where $f_j,g_i\in F$.
We have that \[\begin{array}{rcl} t & =& \tau(\sigma(t))\\ 
&=& \tau(\frac{f(t)}{g(t)})\\ 
&=& \frac{f(\tau(t))}{g(\tau(t)}\\ 
&=& \frac{f(h(t)/k(t))}{g(h(t)/k(t))}\\ 
&=& (\sum f_jh^jk^{-j})/(\sum g_ih^ik^{-i}).\\
 \end{array}\]
 This leads to the following equality in $F[t]$: 
\[ t(\sum_{0\le i\le n}g_ih^ik^{m-i})=\sum_{0\le j\le m} f_jh^jk^{m-j}.\]
The left hand side of the above equality is divisible by $k(t)$ as $m>n$ whereas the right hand side is not as $f_m\ne 0$ and 
$\gcd(h(t),k(t))=1$.   
This contradiction shows that $n=0$ in which case we have already shown that $\deg(f)=1$ 
and so $\sigma$ is a M\"obius automorphism.  

(ii).  As already noted, every automorphism of $F$ extends to a Frobenius automorphism of $R$.  Since any $F$-automorphism of $R$ extends uniquely to an automorphism of $K$, the assertion follows from (i). 
\end{proof}

Let $K_q=\mathbb F_{q} (t)$ for each subfield $\mathbb F_q\subset F$ and let $R_q=R\cap K_q$.  Then $R=\cup R_q$ (where 
the union is over those values of $q$ such that $\mathbb F_q\subset F$). 

 When $\varphi\in Aut(R)$ is a Frobenius automorphism, it is clear that $\varphi$ restricts to an automorphism $\varphi_q$ of $R_q$ for any $q$.   By the above lemma applied to $F=\mathbb F_q$ we see that $\Aut(K_q)$ is 
a finite group and so is $\Aut(R_q)$.

\begin{corollary} \label{locallyfiniteorder}
 We keep the above notations.  
Let $G(R)=\GL(n,R)$ or $\SL(n,R)$.\\
(i) If $\rho=\rho_0\circ \varphi$ where 
$\varphi$ is a Frobenius automorphism and $\rho_0\in \PGL(2,\mathbb F_q)$, then 
$\rho_q:=\rho|_{R_q}$ induces an automorphism $\rho_q:G(R_q)\to G(R_q)$ of finite order. \\
(ii) If $g\in G(R)$, and $\rho\in \Aut(R)$, then the orbit of $g$ under the action of the cyclic group 
$\langle \rho\rangle\subset \Aut(G(R))$ is finite. 
\hfill $\Box$
\end{corollary}

Since $\mathbb F_q[t]\subset  R_q$ for any $q$, it follows that  if $R_q$ is a field, then $R_q=\mathbb F_q(t)$.
Since $R=\cup R_q$ and since $R$ is not a field, there exists a $q$ such that $R_q$ is not a field and so $R_\ell$ is 
also not a field if  $\mathbb F_q\subset \mathbb F_\ell$.

\begin{corollary} \label{Fixed} 
Let $q$ be such that $R_q$ is not a field.  

There exists an $s\in R_q\setminus F$ such that $\rho(s)=s$ for all automorphisms $\rho$ such that $\rho(R_q)=R_q$.

Also the subgroup $\SL_n(\mathbb F_p[s])\subset G(R)$ is element-wise fixed by all such ring automorphisms. 
\end{corollary}
\begin{proof} Let  $\rho\in \Aut(R)$.  Write $\rho=\rho_0.\varphi$ where $\varphi$ is a Frobenius automorphism of $R$ and $\rho_0$ is a M\"obius transformation.  Then $\rho$ restricts to an automorphism of $R_q$ if and only if
$\rho_0$ is defined over $\mathbb F_q$.  
Let $\rho_q$ be the restriction of $\rho$ to $R_q$.   Consequently $\rho(G(R_q))=G(R_q)$ if and only if $\rho_0$ is defined over $\mathbb F_q$.    Conversely, given $\rho_q\in \Aut(R_q)$ it extends to an automorphism of $R$.

Let $\mathfrak R_q=\{\rho\in \Aut(R)\mid \rho(R_q)=R_q\}$.  We have the restriction homomorphism 
$\mathfrak R_q\to \Aut(R_q)$ which is surjective.

By our hypothesis on $R$ and on $q$, 
there exists an element $f\in \mathbb F_q[t]\subset R_q$ which is not invertible.  The orbit $\Omega_f$ of $f$ under the action of $\mathfrak R_q$ is finite since the action factors on $R_q$ through $\Aut(R_q)$, which is a finite group.  
Let $s=\prod_{\rho\in \Aut(R_q)}\rho(f)$ which is the product of all the elements of $\Omega_f$.  
Clearly $s\in R_q$ and is fixed under all automorphisms in $\mathfrak R_q$.

The element $s$ is not  invertible since $f(t)$ is not invertible. 
In particular it is not in $F$.  So $\mathbb F_p[s]$ is isomorphic to a polynomial algebra.  
It is clear that $\SL_n(\mathbb F_p[s])$ is element-wise fixed by $\mathfrak R_q$. 
\end{proof}

The following corollary is immediate from Theorem \ref{aut-omeara} and Remark \ref{remarkonomeara}.

\begin{corollary} \label{transversal}

(i) Any outer automorphism of $\SL_n(R)$ is represented by one of the following automorphisms.  \\
\indent
(a) $\rho:\SL_n(R)\to \SL_n(R)$ induced by an $\mathbb F_p$-algebra automorphism $\rho$ of $R$ or 
$\rho\circ \epsilon$, where $\epsilon$ is the contragredient automorphism.  These are of finite order when $F$ is finite.\\
\indent 
(b) $\iota_{h(\alpha)}\circ \rho, \iota_{h(\alpha)}\circ\rho\circ \epsilon, \alpha\in R^\times$ where 
$h(\alpha)=\delta(I_{n-1},\alpha)\in \GL_n(R)$. 
These are  of finite order when $F$ is finite.

\noindent
(ii) Any outer automorphism of $\GL_n(R)$ is represented by one of the following automorphisms:\\
\indent (a) an automorphism 
$\rho$, induced by an automorphism of $R$ or $\rho\circ \epsilon$ where $\epsilon$ is the contragredient automorphism. These are of finite order when $F=\mathbb F_q$.\\
\indent 
(b) $ \mu_\chi\circ \rho, \mu_\chi\circ \rho\circ \epsilon,$ where $\chi$ is a suitable character $\chi:\GL_n(R)\to R^\times$. 

\noindent
(iii) Every automorphism of $\SL_n(R)$ extends to an automorphism of $\GL_n(R)$.  \hfill $\Box$
\end{corollary}

\subsection{Proof of Theorem \ref{fpbar} for $F$ a finite field} \label{prooffornge3}
We shall now prove Theorem \ref{fpbar} when $F=\mathbb F_q$ and $n\ge 3$.  The more general case when $F\subset \bar{\mathbb F}_p$ will be 
established in \S \ref{mainresults}.  
In order to emphasise this 
restriction we shall use $A$ instead of $R$ to denote a ring such that $\mathbb F_q[t] \subset A \subsetneq \mathbb F_q(t)$.

Let $G=\GL_n(A)$ or $\SL_n(A)$.  To show that $G$ has the $R_\infty$-property,  
it suffices to show that $R(\phi)=\infty$ for a set 
$\mathcal S$ of representatives of the outer automorphisms of $G$. We take $\mathcal S$ to be as in Corollary \ref{transversal}.

Consider the automorphism $\rho:G\to G$ induced by a ring automorphism $\rho:A\to A$.   
Let $S=\mathbb F_p[s]\subset A$ be as in Corollary \ref{Fixed}.  Then $S$ is 
 contained in the subring $\fix(\rho)\subset A$ and the group $G(S)\subset G(A)$ is element-wise fixed by $\rho$.  
 
  Set $x_m=e_{12}(s^m)e_{21}(-s^m)\in \SL_2(S)$ so that $x_m=
 \bigl(\begin{smallmatrix} 1-s^{2m} & s^m \\ -s^m & 1\end{smallmatrix} \bigr)$.  
 We observe that $\rho(x_m)=x_m=e_{12}(s^m)\epsilon(e_{12}(s^m))$ and that the $x_m$  satisfy the 
 polynomial $X^2+(s^{2m}-2)X+I_2=0$.   
We regard $x_m$ as also as an element of $\SL_n(S)$ by identifying it 
 with the block diagonal matrix $\delta(x_m,I_{n-2})$.  
These elements will play an important role in the our proofs as they will be shown to be in pairwise distinct $\phi$-twisted conjugacy classes 
 for many automorphisms of $G$.  The following lemma will play a crucial role in our proof.

\begin{lemma}\label{trace}
 Let $A$ be a ring such that $\mathbb F_q[t] \subset A \subsetneq \mathbb F_q(t)$. Fix $r\ge 1$. The elements $x_m=e_{12}(s^m)e_{21}(-s^m),\in \SL_n(A), m\ge 1,$ 
 are such that $\tr(x_ m^r)$ are pairwise distinct.  
\end{lemma}
\begin{proof} 
Set $x:=e_{12}(u)e_{21}(-u)=\bigl(\begin{smallmatrix} 1-u^2 & u \\ -u& 1\end{smallmatrix} \bigr)\in \GL_2(\mathbb F_p[u])$.  
We see that $\tr(x)=2-u^2, \tr(x^2)=2-4u^2+u^4$.   As the characteristic polynomial of $x$ is $X^2-(2-u^2)X+1$, we obtain 
the relation $\tr(x^r)=(2-u^2)\tr(x^{r-1})-\tr(x^{r-2})$ for any $r\ge 3$.   It follows by induction that $\tr(x^r)$ is a polynomial in $u$ of degree 
$2r$ with leading coefficient $(-1)^r\in \mathbb F_p$.    

The last assertion still holds when $x$ is viewed as an element of $\GL_n(\mathbb F_p[u]), n\ge 3$.  Applying this to the elements $x_m\in \GL_n(A)$ 
defined above, $\tr(x_m^r)\in S=\mathbb F_p[s]\subset A$ is a polynomial in $s$ of degree $2rm$.   Hence $\tr(x_m^r), m\ge 1,$ 
are pairwise distinct.  
\end{proof}

We are now ready to prove Theorem  \ref{fpbar} when $F$ is a finite field.   

\noindent
\begin {proof}
The proof will depend on the type of automorphism as listed in 
Corollary \ref{transversal}. The symbol $\rho$ will always denote an automorphism of $G$ induced by a 
ring automorphism of $A$, $\epsilon$, the contragredient, $\iota_g$, the conjugation by a $g\in \GL_n(A)$, etc. 

First we consider the case when $G=\GL_n(A)$.

{\it Type} $\rho$:   Note that $x_m\in \fix(\rho)$.  
Taking $r=o(\rho)$ in Lemma \ref{trace}, we 
 see, by 
Lemma \ref{fixedgroup}, that the $x_m^r,m\ge 1,$ are in pairwise distinct $\rho$-twisted conjugacy classes and so $R(\rho)=\infty$.

{\it Type} $\rho\circ \epsilon$:  Let $\theta=\rho\circ \epsilon$.   
Since $\rho \circ \epsilon=\epsilon\circ \rho$ and since $\epsilon^2=id$ we have 
$\theta^2=\rho^2$.
We shall show that $e_{12}(s^k), e_{12}(s^m)$ are not in the same $\theta$-twisted 
conjugacy class if $m>k\ge 1$. 

Suppose that $e_{12}(s^m)=z e_{12}(s^k) \theta(z^{-1})$.   Applying $\theta$ to both sides we obtain 
$e_{21}(-s^m)=\theta(z) e_{21} (-s^k) \theta^2(z^{-1})$.   Multiplying the two equations and using $\theta^2=\rho^2$ we obtain that 
\[x_m=e_{12}(s^m)e_{21}(-s^m)=z e_{12} (s^k) e_{21}(-s^k)\rho^2(z^{-1})=zx_k\rho^2(z^{-1}).\eqno(2)\]   
That is, $x_k,x_m$ are in the same $\rho^2$-twisted conjugacy class.  By what has been shown in the case of `{\it Type $\rho$}',  this is a contradiction. It follows that 
$R(\theta)=\infty$.

{\it Type} $\theta=\mu_\chi\circ \rho$:  

Let $\theta=\mu_\chi\circ \rho$ and let $r=o(\rho)$.  We claim that the elements $x_m^r\in \SL_n(A), 
m\ge 1$, are in pairwise distinct $\theta$-twisted conjugacy classes.  Suppose that 
$x_k=z x_m \theta (z^{-1})$ for some $z\in \GL_n(A)$.  Note that $\theta(z^{-1})=\rho(z^{-1})u$ where 
$u:=\chi(\rho(z^{-1}))I_n$.  

Thus $x_k=zx_m\rho(z^{-1})u$.   
Since $u$ is in the centre of $\GL_n(A)$, for any $j\ge 1$, applying $\theta $ 
repeatedly, we obtain that for any $j\ge 0$, 
$x_k=\theta^j(x_k)=\rho^j(z) x_m \rho^{j+1}(z^{-1})u_j$ for a suitable scalar matrix $u_j$ in $\GL_n(A)$.  Setting  $r:=o(\rho)$,   
we are led to the equation $x_k^r=zx_m^rz^{-1}v$ for some scalar matrix $v\in \GL_n(A)$. Assuming $v=\delta(a,\ldots,a)$ and taking determinant on the both sides of the above equation, we obtain that $a^n=1$, i.e., $a$ is a torsion element in $A$ and hence $v$ is a scalar matrix in $\GL_n(\mathbb{F}_q)$. Now, we take trace on both sides of the equation $x_k^r=zx_m^rz^{-1}v$ and get $\textrm{tr}(x_k^r)=a\textrm{tr}(x_m^r)\in \mathbb F_q[s]$.  This is a contradiction to Lemma \ref{trace} as the degree of $\textrm{tr}(x_j^r)$ as a 
polynomial in $s$ equals $2jr$.  This shows that $R(\mu_\chi\circ \rho)=\infty$.

{\it Type} $\mu_\chi\circ \rho\circ \epsilon$: The proof that $R(\mu_\chi\circ \rho\circ \epsilon)=\infty$ uses $e_{12}(s^m)\in \SL_n(A)$ and  is similar to the proof for the type $\rho\circ \epsilon$, just 
as the above proof for $\mu_\chi\circ \rho$ parallels the proof for type $\rho$.

This completes the proof that $\GL_n(A)$ has the $R_\infty$-property for $n\ge 3$.  

It remains to consider the case of automorphisms of $\SL_n(A)$ as in Corollary \ref{transversal}(i).    

{\it Type} $\iota_h\circ \rho$:
Consider an automorphism $\phi$ of $\SL_n(A)$ of the form $\phi=\iota_h\circ \rho$ with $h=h(a)\in H$ with $a\in A^\times$ as in Corollary \ref{transversal}(ii).  
Suppose that $k$ and $m$ are distinct but $x_k=z x_m\phi(z^{-1})=z x_m h\rho(z^{-1})h^{-1}$.  So $x_kh$ and $x_mh$ are $\rho$-twisted conjugates in $\GL_n(A)$. We apply Lemma \ref{fixedgroup}(i) to $\rho$.  Setting $r=o(\rho)$ we obtain, 
$\prod_{0\le j<r}\rho^j(x_kh)$ and $\prod_{0\le j<r}\rho^j(x_mh)$ are conjugates.  Since $x_k$ and $\rho^j(h)=h(\rho^j(a))$ commute, we obtain that 
$\prod_{0\le j<r}\rho^j(x_kh)=x_k^rh(\prod \rho^j(a))$ and the same holds when $k$ is replaced by $m$. Now, $\tr(x_k^rh(\prod \rho^j(a)))=n-3+\prod \rho^j(a)+\tr(x_k^r)$ and so $\tr(\prod_{0\le j<r}\rho^j(x_kh))=\tr(\prod_{0\le j<r}\rho^j(x_mh))$ implies $\tr(x^r_k)=\tr(x^r_m)$, contradicting Lemma \ref{trace}.

{\it Type} $\iota_h\circ \rho\circ \epsilon$:
Finally, it remains to consider automorphisms of $\SL_n(A)$ of the form $\psi:=\iota_h\circ \rho\circ \epsilon$ with 
$h=h(a), a\in A^\times$, as in Corollary \ref{fixedgroup}.  
We assert that $e_{12}(s^m), m\ge 1,$ are in pairwise distinct $\psi$-twisted conjugacy classes.
Suppose that there exist positive integers $k,m$ such that 
\[e_{12}(s^m)=z e_{12}(s^k)\psi(z^{-1}).\]  
Applying $\psi$ to both sides of this equation we obtain \[ e_{21}(-s^m) =\psi(z) e_{21}(-s^m)\psi^2(z^{-1}).\]
Multyplying the two sides of the two equations and using $x_m=e_{12}(s^m)e_{21}(-s^m)$ we have 
\[x_m=z x_k \psi^2(z^{-1}).\]  Since, by Example \ref{thetasquare}, $\psi^2=\iota_{h\rho(h^{-1})} \circ \rho^2=\iota_{h(a\rho(a^{-1}))}\circ \rho^2$, it is of the previous type, namely, ``{\it Type $\iota_h\circ \rho$}".  
So  $x_m,x_k$ are not $\psi^2$-twisted conjugates unless  $m= k$.   Therefore we conclude 
that that $R(\psi)=\infty$, completing the proof.

\end{proof}


 \section{Proofs of Theorem \ref{fpbar} and Theorem \ref{sandwich}} \label{mainresults} 

Let $R$ be a ring such that $F[t]\subset R\subsetneq K:=F(t)$ where $F$ is a subfield of $\bar {\mathbb F}_p$ 
and $t$ is a variable.

Let $G$  denote one of the groups $\GL_n(R), \SL_n(R), n\ge 3$.  
The proof of Theorem \ref{fpbar} for $F\subset \bar{\mathbb F_p}$ is similar to the special case when $F=\mathbb F_q$.   In fact, most of the proof in the general case reduces to the special case.  For this reason we shall omit most of the details.  

Let $\rho \in Aut(R)$ be defined over $\mathbb{F}_q$ for some $q$. 

By our choice of $q$, if $\mathbb F_q\subset \mathbb F_\ell \subset F$, we see that 
$R_\ell:= R \cap \mathbb F_\ell(t)\subset F(t)$ is stable by $\rho$.  Also, since $R$ is not a field, we may (and do) assume that 
$R_\ell $ is not a field.  
Note that $R=\cup R_\ell$ where the union is over such values of $\ell $ that $ R_\ell$ is not a field.  

Consequently, the groups $G_\ell:=\GL_n(R_\ell), \SL_n(R_\ell)$ are stable by 
$\rho$ and Theorem \ref{fpbar} holds for each of them.   
We observe that $G$ is the union of the groups  $G_\ell$.

\subsection{Proof of Theorem \ref{fpbar}}
We are now ready to prove Theorem \ref{fpbar}.  

\noindent 
{\it Proof of Theorem \ref{fpbar}.}
As observed already,  with notations from \ref{transversal}, 
we need only consider the automorphisms $\phi=\mu_\chi\circ \rho, \mu_\chi\circ \rho\circ \epsilon$ when $G=\GL_n(R)$,  
and, when $G=\SL_n(R)$, the automorphisms $\phi=\iota_h\circ \rho, \iota_h\circ \rho\circ \epsilon$ where $h=h(a)=\delta(I_{n-1},a), a\in R^\times$.

Let $\rho \in Aut(R)$ be defined over $\mathbb{F}_q$ for some $q$. Let $s \in R_q$ be as in Lemma \ref{Fixed} and $x_m=e_{12}(s^m)e_{21}(-s^m)\in G_q=\GL_n(R_q), m\ge 1,$ be as in Lemma \ref{trace} where $R_q=R \cap \mathbb F_q(t)$.  
Then $x_m\in \fix(\rho)$.  
Suppose that there exists an element $z\in G$ such that 
$x_k=z x_m\rho(z^{-1})$ with $k\ne m$.   There exists $\ell=q^d=p^{de}$ a sufficiently large power of $q$ such that 
$\mathbb F_\ell\subset F$ and  $x_k,x_m, z\in G_\ell$.   Then $\rho^N|_{G_\ell}=id$ where $N:=de$.  
This implies, by Lemma \ref{fixedgroup}  that $x^N_m$ and $x^N_k$ are conjugates in $G_\ell$.   
This contradicts Lemma \ref{trace} and we conclude that $R(\rho)=\infty$.   The proof for $\rho\circ \epsilon$ is 
similar to the proof of the corresponding type of automorphism in Theorem \ref{fpbar} for $n\ge3$ given in \S\ref{prooffornge3}.

Now let $\phi=\mu_\chi\circ \rho$.   
    
Suppose that $x_k\sim_\phi x_m$ for some $k\ne m$.
Let $z\in G$ such that $x_k=zx_m \phi(z^{-1})=zx_m\rho(z^{-1})uI_n$ where $u=\chi (\rho(z^{-1}))\in R^\times$.  Then there exists an $\ell=q^d=p^{de}$ for a sufficiently large $d$ so that $z\in G_\ell$.  Then $\rho^{de}|_{G_\ell}=id$.  
Applying $\rho$ repeatedly to 
both sides of this equation, we obtain $x_k= \rho^{j}(z) x_m \rho^{j+1}(z^{-1}).u_jI_n$ 
for $j\ge 1$ for suitable $u_j\in R^\times$.   
Multiplying these equations in order for $0\le j<de$ and using the fact that $\rho^{de}(z^{-1})=z^{-1}$ 
we obtain that $x_k^{de}=z x_m^{de} z^{-1}. vI_n$ for some $v\in R^\times$. First taking determinant on both sides of the equation, we observe that $v$ is a torsion element of $R$ and so $v \in F^\times$. Now taking trace on both sides we get 
$\textrm{tr}(x_k^{de})=v\textrm{tr}(x_m^{de})$.  This is a contradiction since $v\in F$ and the degrees of traces of $x_k^{de}, x_m^{de}$ as polynomials in 
$s$ are $2kde, 2mde$ respectively which are unequal as $k\ne m$.  Hence we conclude that $R(\phi)=\infty$ in this case. 
 
The proof is similar when $\phi=\mu_\chi\circ \rho\circ \epsilon$.  This completes the proof when $G=\GL_n(R)$.  

When $G=\SL_n(R)$, we need to show that $R(\phi)=\infty$  when $\phi=\iota_h\circ \rho, \iota_h \circ \rho\circ \epsilon$ where $h=h(a)=\delta(I_{n-1}, a),a\in R^\times$.   We choose $q=p^e$ so that $\rho$ restricts to $G_q$ and hence to $G_\ell$ for all $\ell=q^r$.  We choose $\ell$ so that $R_\ell$ is not a field.  The rest of the proof is as in the proof 
of  Theorem \ref{fpbar} for the automorphisms of $\SL_n(R_\ell)$ of the corresponding types, given in \S \ref{prooffornge3}.  The details are left to the reader. \hfill $\Box$


\subsection{Proof of Theorem \ref{sandwich}}\label{sandwichedgroups}
 
We begin by describing the multiplicative group $R^\times$ of all invertible elements of $R$ and the action of $\Aut(R)$ 
on it.

Let $\mathcal B$ be the set of all monic irreducible polynomials in $F[t]$ which are invertible in $R$ and let $S\subset F[t]$ 
be the multiplicative set generated by $\mathcal B$.  Then $R=S^{-1}F[t]\subset F(t)$.   
Any  $f(t)\in R^\times $ 
has a unique factorization $f(t)=af_1^{n_1}\cdots f_k^{n_k}$ where 
elements $f_j\in \mathcal B, n_j$ non-zero integers, and $a\in F^\times$.   
It follows that $R^\times $ is isomorphic 
to $F^\times \times U$ where $U$ is a free abelian group with basis $\mathcal B$.  
Any automophism of $R^\times$ maps $F^\times$ to itself and induces an automorphism of $R^\times/F^\times \cong U$. 

There is a bijective correspondence between subgroups $D\subset R^\times $ and subgroups $H(D)\subset \GL_n(R)$ that 
contain $\SL_n(R)$ where $H(D)=\{g\in \GL_n(R)\mid \det g\in D\}$.  
Recall that  $\SL_n(R)$ is perfect 
in view of the fact that $R$ is a Euclidean domain.  It follows
that $\SL_n(R)=[H(D),H(D)]=[\GL_n(R),\GL_n(R)]$. 
Hence $\SL_n(R)$ is characteristic in $H(D)$ for all $D\subset R^\times$.

\begin{lemma} \label{centralizer}
Suppose that $g\in H=H(D)$ commute with every element of $\SL_{n-1}(R)$.  Then $g=ah(b)$ 
for some $a,b\in R^\times$ such that $a^{n}b\in D$. 
\end{lemma}
\begin{proof}
Suppose that $g$ is not a diagonal matrix, say $g=(g_{ij})$ with $g_{k,m}\ne 0$ where $k\ne m$.  Then 
$e_{1k}(1)ge_{1k}(-1)\ne g\ne e_{1m}(1) ge_{1m}(-1)$.  Since at least one of $k,m$ is less than $n$ we get a 
contradiction.
Now suppose that $g$ is a diagonal matrix $g=\delta(a_1,\ldots, a_n)$. 
If $a_i\ne a_j$ for some $i<j<n$, then $ge_{ij}(1)g^{-1}\ne e_{ij}(1)$ and the lemma follows.
\end{proof}

Let $H=H(D)$ and let 
$\theta:H\to H$ be an automorphism.  Then $\theta$ restricts to an automorphism $\theta':\SL_n(R)\to \SL_n(R)$. 
Replacing $\theta $ by $\theta\circ \iota_g$ for a suitable $g\in \SL_n(R)$ if necessary, we may (and do) assume that 
$\theta'=\iota_h\circ \rho\circ \eta$ where $h=h(a)=\delta(I_{n-1},a), a\in R^\times$, $\rho: \SL_n(R)\to \SL_n(R)$ is induced by an automorphism 
$\rho: R\to R$ of the ring $R$ and $\eta$ belongs to the cyclic group $\langle \epsilon\rangle\cong \mathbb Z/2\mathbb Z$.
Then $\theta'(\SL_{n-1}(R))=\SL_{n-1}(R)$.   
Evidently, $\theta'$ extends to a 
automorphism $\phi:\GL_n(R)\to \GL_n(R)$.  Any  extension $\phi$ equals $\mu_\chi\circ \iota_{h(a)} \circ \rho\circ \eta$ where $\mu_\chi$ is a homothety automorphism associated to a character $\chi:\GL_n\to R^\times$.  
It is not clear that $\theta\in \Aut(H)$ admits an extension $\phi:\GL_n(R)\to \GL_n(R).$   
When it does, following proposition guarantees that $R(\theta)=\infty$.

\begin{proposition}
Let $\theta\in \Aut(H)$ and suppose that $\theta=\phi|_{H}$ for some $\phi\in \Aut(\GL_n(R))$.  Then $R(\theta)=\infty$.
\end{proposition}
\begin{proof}
We keep the above notation.  By the above discussion, 
we need only consider the case $\theta'=\iota_{h(a)}\circ \rho\circ \epsilon$ or $\iota_{h(a)}\circ \rho$ in $\Aut(\SL_n(R))$. 
 First suppose that $\theta'=\iota_{h(a)}\circ \rho$. 
Choose  a non-zero, non-unit element $s\in R$ fixed by $\rho$.   Such an element exists by Lemma \ref{Fixed}. 
Choose $q$ so that $s\in R_q=R\cap \mathbb F_q[t]$.  Let $r$ be the order of 
the automorphism $\rho|_{R_q}$.    Since, the commutation relation  $\mu_\chi\circ\iota_h=\iota_h\circ \mu_\chi$ holds, we have $\phi=\mu_\chi\circ \iota_{h(a)} \circ \rho=\iota_{h(a)}\circ \psi$ where 
$\psi:=\mu_\chi\circ \rho$.  

 Consider the elements $x_m=e_{12}(s^m)e_{21}(-s^m)$, $m\ge 1$.  
 We claim that $x_k$ and $x_m$ are in distinct $\theta$-twisted 
conjugacy classes if $k\ne m$.   
Assume that $[x_m]_\theta =[x_k]_\theta$.  It follows that $[x_k]_\phi=[x_m]_\phi$.  This implies that $[x_k h(a)]_\psi =[x_m h(a)]_\psi$.  
Proceeding as in the case of {\it Type $\mu_\chi\circ \rho$} in the proof of Theorem \ref{fpbar} for finite fields, we obtain 
that \[x^r_m h(u) =z x_m^r vh(u). z^{-1}\] where $v\in R^\times, u=\prod_{0\le j<r}\rho^j(a)$.   Taking determinants on both sides, we obtain 
that $u=v^n u$ and so $v^n=1$.   Raising to the $n$-th power, we obtain 
$x_m^{rn} h(u^n)=zx_m^{rn}h(u^n)z^{-1}$.  Taking trace on both sides of the last equality we obtain that $\tr(x_m^{rn})=\tr(x_k^{rn})$. 
This 
contradicts Lemma \ref{trace} if $k\ne m$.   Hence $R(\psi)=\infty$.

The proof in the case when $\psi=\mu_\chi\circ \rho\circ \epsilon$ is similar and omitted.  
\end{proof}

Write 
$\rho=\rho_0\circ \varphi$ where $\rho_0$ is an $F$-automorphism and $\varphi$ is a Frobenius automorphism.
We choose $q$ so that $\rho_0$ is defined over $\mathbb F_q$ and that $a\in \mathbb F_q$.   We further assume that $R_q=R\cap \mathbb F_q(t)$ is not a field.    All these conditions can be met so long as $q$ is sufficiently large (and 
$\mathbb F_q\subset F$).

Recall that if $\rho\in \mathfrak R_q$, then $\rho|_{R_q}$ has finite order and so the same is true of the induced automorphism of $\SL_n(R_q)$.

We are ready to prove Theorem \ref{sandwich}

\begin{proof}  Let $\theta\in \Aut(H)$.  Replacing $\theta$ by $\iota_g\circ \theta$ if necessary, we may (and do) 
assume that $\theta'=\theta|_{\SL_n(R)}$ equals $\iota_{h(a)}\circ \rho$ or $\iota_{h(a)}\circ \rho\circ \epsilon$ 
where $a\in R^\times$, $\rho\in \Aut(\SL_n(R))$ is a ring automorphism $\rho:R\to R$ and $\epsilon $ is the contragredient. 
We first consider the case $\theta'=\iota_h\circ \rho$ where $h=h(a)$.  We choose $q$ so that $R_q$ is not a 
field and $\rho$ restricts to an automorphism of $R_q$.   

Let $x_m\in \SL_n(R_q)\subset H,m\ge 1,$ be as in Lemma \ref{trace}.   Suppose that $m>k\ge 1$ and that 
\[x_m=zx_k\theta(z^{-1}).\eqno(3) \]   
Substituting $z=yh(c), $ we obtain $\theta(z)=\theta'(y)\theta(h(c))= h(a)\rho(y)h(a^{-1})\theta(h(c))$.   
Since $h(c)$ commutes 
with $\SL_{n-1}(R)$, Lemma \ref{centralizer} implies that $\theta(h(c))=b_1h(c_1)$ for some $b_1, c_1\in R^\times$.  
Taking determinants on both sides of Equation (3), we obtain $\det(z)=\det (\theta(z))$ which implies that 
have $\theta(h(c))=b_1h (b_1^{-n}\rho(c))$ for some $b_1\in R^\times$.

Since $\det(x_m)=\det(x_k)=1=\det(y)$, Equation (3) implies that  $c=\det (z)=\det (\theta(z))=b_1^n. b_1^{-n}\rho(c)$ 
and so $\rho (c)=c$. It follows that $\theta(h(c))=b_1h(b_1^{-n}c)$.
  
Since $x_k$ commutes with $h(b)$ for all $b\in R^\times$,  we get 
$zx_k\theta(z^{-1})=yh(c)x_k b_1^{-1}h(b_1^nc^{-1})\theta'(y^{-1})=yh(c) x_k b_1^{-1}h(b_1^nc^{-1}) h(a) \rho(y) h(a^{-1})=
yx_k b_1^{-1} h(ab_1^n)\rho(y)h(a^{-1})$.   
Hence we obtain that 
\[ x_mh(a) =b_1^{-1} yh(ab_1^n)x_k\rho(y^{-1}).\eqno(4)\]
Applying $\rho^j$ to both sides and using $\rho(x_k)=x_k,\rho(x_m)=x_m, \rho(h(u))=h(\rho(u))$, we get
\[x_mh(\rho^j(a))=\rho^j(b_1^{-1}) \rho^j(y) h(\rho^j(ab_1^n)) x_k\rho^{j+1}(y^{-1}).\eqno(5)\] 
Replacing $q$ by a finite extension field in $F$ if necessary, we may (and do) assume that 
$y,c,a, b_1$ are all in $\mathbb F_q$. 
Multiplying these successively as $j$ varies from $0$ to $r-1$ where $r:=o(\rho|_{R_q})$, we obtain that 
\[x_m^r\prod_{0\le j<r}  h(\rho^j(a))=\prod_{0\le j<r} \rho^j(b_1^{-1}) y (\prod_{0\le j<r} h(\rho^j(ab_1^n)))x_k ^ry^{-1}\eqno(6)\]
since $\rho(x_k)=x_k,\rho(x_m)=x_m$. 

To simplify notations, set $ \beta:=\prod_{0\le j<r}\rho^j(b_1^{-1})$, 
$u:=\prod_{0\le j<r} \rho^j(a), 
v:=\prod_{0\le j<r}
\rho^j(ab_1^n)=\beta^{-n} u$.  
Equation (6) says that the matrices $x_m^rh(u)$ and $\beta x_k^rh(v)$ are similar.  Note that these are block diagonal matrices with block sizes 
$2, 1,\ldots, 1$.   Recall the definition $x_m=e_{12}(s^m)e_{21}(-s^m)=\delta(A_m,I_{n-2})$ where 
$A_m=\bigl(\begin{smallmatrix} 1-s^{2m} & s^m\\-s^m&1\end{smallmatrix}\bigr)$.   
The characteristic polynomials of $x_m^rh(u)$ and $\beta x_k^r h(v)$ are:  
$(X^2-\lambda_m X+1)(X-1)^{n-3}(X-u)$ and $(X^2-\beta\lambda_k X+\beta^2)(X-\beta )^{n-3}(X-\beta v)$ 
respectively, where $\lambda_m=\tr(A_m^r)$.   If at least one of $(X^2-\beta \lambda_k+\beta^2)$ or 
$(X^2-\lambda_m X+1)$ is irreducible (in $R$), then so is the other and the two must be equal. This implies 
that $\beta=\pm 1$ and $\lambda_m=\pm \lambda_k$.  This is a contradiction to Lemma \ref{trace}.  It follows 
that both of them factor into linear factors in $R$.   Let $\alpha_k^{\pm 1}, \alpha_m^{\pm 1}$ be the eigenvalues 
of $A_k,A_m$ respectively.  We do not assume that these eigenvalues are in $R$. However, 
$\alpha_k^r,\alpha _k^{-r}, \alpha_m^r,\alpha_m^{-r}$ are all in $R$.  By Lemma \ref{trace}, $\alpha_k^r\ne \alpha_k^{-r}$ 
and $\alpha_m^r\ne \alpha_m^{-r}$.  

Suppose that $\beta\in R^\times$ is a torsion element. Say $\beta^\ell=1$.  Then we may raise to the $\ell$-th power 
both sides of Equation (6) and obtain the same equation in which $r$ is replaced by $r\ell$ and $\beta$ by $1$. 
Consequently, we obtain that traces of $A_m^{r\ell}$ and $A_k^{r\ell}$ are equal, contradicting Lemma \ref{trace}.
So, we must have that $\beta\notin F^\times$.   

Then the eigenvalues of $x_m^rh(u)$ (resp. $ \beta x_k^rh(v)$) form the following multiset 
$\mathcal M: \alpha_m^r,\alpha^{-r}_m, 1 $ with multiplicity $(n-3)$, and $u$ (resp. the multiset 
$\mathcal K:  \beta \alpha_k^r,\beta\alpha^{-r}_k, \beta $ with multiplicity $n-3$, and $\beta v$).

The rest of the proof will involve case consideration depending on the value of $n$.

{\it Case 1:} Suppse that $n\ge 5$.   Since $1$ occurs in $\mathcal M$ with multiplicity at least $2$, it occurs in $\mathcal K$ with the same multiplicity.   Since $\beta \alpha_k^r\ne \beta \alpha_k^{-r}$, at most one of them can 
equal $1$, and so we 
must have (a) $\beta=1$ or (b) $\beta v=1$.    
Since $\beta\notin F^\times$, we have that $\beta\ne 1$.    

Suppose that $\beta v=1$ (and $\beta \neq 1$).  Since $v=u\beta^{-n }$, we have $u=\beta^ {n-1}$.  Since $1$ occurs in $\mathcal M$ with multiplicity at least $2$, we must have 
$\beta \alpha_k^r=1$ or $\beta\alpha_k^{-r}=1$.  In either case, we have 
$\mathcal K=1, \beta^2, \beta, 1$ where $\beta$ occurs with multiplicity $(n-3)$.  Since $\alpha_m^{r}.\alpha_m^{-r}
=1$, we must have two distinct elements of $\mathcal K$ which are reciprocals of each other.  This means that 
$\beta^3=1$ or $\beta^2=1$ and so $\beta\in F^\times$, a contradiction.

{\it Case 2:} Suppose that $n=4$.  We know that $\beta\ne 1$.  Since $1$ occurs in $\mathcal M$, one of the following must hold:
(i) $\beta v=1$, or (ii) $\beta \alpha_k^r=1$, or (iii) $\beta\alpha_k^{-r}=1$.  
 
{\it Subcase} (i): Suppose that $\beta v=1$.  
 It follows that $u=\beta^{n-1}=\beta^3$.   Since $\beta$ occurs in 
$\mathcal K$, it should occur in $\mathcal M$.  Since $u\ne \beta\ne 1$, we must have 
$\alpha^{r}_m=\beta$ or $\alpha^{-r}_m=\beta$.  In any case we have $\mathcal M=\{\beta,\beta^{-1}, 1,\beta^3\}$.
It follows that $\{\beta\alpha_k^r,\beta\alpha_k^{-r}\}=\{\beta^3,\beta^{-1}\}$.   So we have 
$\alpha_k^r+\alpha_k^{-r}=\beta^2+\beta^{-2}$ and $\alpha_m^r+\alpha_m^{-r}= \beta+\beta^{-1}$. This leads to the 
relation $(\tr(A_m^r))^2-\tr(A_k^r)-2=0$.  Since $\tr(A_m^r), \tr(A_k^r)\in F[s]$ are  polynomials in $s$ of degree $2mr, 2kr$ 
respectively, and since $m>k$, the last equation cannot hold.  So we are led to the conclusion that $\beta v\ne 1$.

{\it Subcase} (ii):  Suppose that $\beta \alpha_k^r=1$.  Then $\beta \alpha_k^{-r}=\beta^2$ and $\mathcal K=\{1,\beta^2,\beta, \beta v\}$.  Since $\alpha_m^r,\alpha_m^{-r}\in \mathcal K$ and $\alpha_m^r\ne 1$, the product of some pair of  elements of $\{\beta^2,\beta,\beta v\}$ should equal $1$.  If $\beta^3=1$, then $\beta \in F^\times$, a contradiction.  Then there are two possibilities: (a) $\{\alpha_m^r,\alpha_m^{-r}\}=\{\beta,\beta v\}$ or (b) $\{\alpha_m^r,\alpha_m^{-r}\}=\{\beta^2,\beta v\}$. If $\{\alpha_m^r,\alpha_m^{-r}\}=\{\beta,\beta v\}$, then 
$\alpha_m^r+\alpha_m^{-r}=\beta+\beta^{-1}$.  Since $\beta \alpha_k^r+\beta\alpha_k^{-r}=1+\beta^{2}$, we obtain 
that $\alpha_k^r+\alpha_k^{-r}=\beta^{-1}+\beta=\alpha_m^r+\alpha_m^{-r}$.  Thus $\tr(A_m^r)-\tr(A_k^r)=0$ in 
$F[s]$, again a contradiction as $m>k$. If $\{\alpha_m^r,\alpha_m^{-r}\}=\{\beta^2,\beta v\}$, then $\alpha_m^r+\alpha_m^{-r}=\beta^2+\beta^{-2}$. Again, we have $\alpha_k^r+\alpha_k^{-r}=\beta^{-1}+\beta$. The last two equations imply that $(\tr(A_k^r))^2-\tr(A_m^r)-2=0$. Since $\tr(A_m^r), \tr(A_k^r)\in F[s]$ are  polynomials in $s$ of degree $2mr, 2kr$ 
respectively, we have $m=2k$. So $\{x_k \mid k\in \mathbb N, k\equiv 1\mod 2\}$  is an infinite set of elements in $H$ that are in pairwise distinct $\theta$-twisted conjugacy classes. 

{\it Subcase} (iii):   Suppose that $\beta\alpha_k^{-r}=1$.  Then $\beta\alpha^r_k=\beta^2$ and so 
$\mathcal K=\{1,\beta^2,\beta,\beta v\}$.  Rest of the proof is exactly as in subcase (ii). 

This completes the proof in the case when $\theta'=\theta|_{\SL_n(R)}=\iota_h\circ \rho$.

Next consider the case  $\theta'=\iota_h\circ \rho \circ \epsilon$ where $h=h(a), a\in R^\times$, $\rho$ is is induced by a ring automorphism $\rho:R\to R$, and $\epsilon$ is the contragredient.   
As in the proof of Theorem \ref{fpbar} for the type $\iota_h\circ \rho\circ \epsilon$, we consider the sequence of elements 
$\{e_{12}(s^m)\}_{m\ge 1}$ and show that an infinite subsequence of this sequence are in pairwise distinct 
$\theta$-twisted conjugacy classes.  Here again, $s\in R$ is a non-zero non-invertible element fixed by $\rho$.  

Suppose that $m>k\ge 1$ and the following holds: 
\[ e_{12}(s^m)=z e_{12}(s^k)\theta(z^{-1})  \eqno (7)\]
where $z=yh(c)$ with $y\in \SL_n(R), c\in D$.   
Applying $\theta$ to both sides of the above equation we get 
\[ e_{21}(-s^m)=\theta(z) e_{21}(-s^m) \theta^2(z^{-1}).\eqno(8)\]
We multiply each side of Equations (7) and (8) to obtain, using $e_{12}(s^m)e_{21}(-s^m)=x_m$,  
\[ x_m=z x_k\theta^2(z^{-1}). \eqno(9)\]
From Example \ref{thetasquare}  we have $\phi:=\theta^2= \iota_{h(a\rho(a^{-1}))}\circ \rho^2$.

Therefore, setting $\alpha=a\rho(a^{-1})$ we obtain that $x_m$ and $x_k$ are $\phi$-twisted conjugates.

Suppose that $n\ge 4$ or $D\subset F^\times$.  
By what has been established already, the sequence $\{x_k\}_{k\ge 1}$ contains an infinite subsequence 
whose terms are in pairwise 
distinct $\phi$-twisted conjugacy classes.  Therefore the sequence $\{e_{12}(s^k)\}_{k\ge 1}$ contains an 
infinite subsequence whose terms are in pairwise distinct $\theta$-twisted conjugacy classes. 
This completes the proof.


\end{proof}
 
 Let $n=3$.  Suppose that $u=\beta v$.    Since $v=\beta^{-3}u$, we obtain that $u=\beta ^{-2}u$ and so $\beta^2=1$ a contradiction.  

There remains only the following four possibilities are :\\
 (a) $u=\beta\alpha_k^r, \alpha_m^r=\beta v,\alpha_m^{-r}=\beta\alpha^{-r}_k$,\\
(b)  $u=\beta\alpha_k^{-r}, \alpha_m^r=\beta v, \alpha_m^{-r}=\beta\alpha_k^{r}$, \\
(c) $u=\beta \alpha_k^{r}, \alpha_m^r=\beta \alpha_k^{-r}, \alpha_m^{-r}=\beta v$, \\
(d) $u=\beta \alpha_k^{-r}, \alpha_m^r=\beta \alpha_k^r,\alpha_m^{-r}=\beta v$.\\

Each of these possibilities appears to be consistent with the relation $v=\beta^{-3}u$ and we have not been 
able to show that $R(\theta)=\infty$ in these cases.  

We end this section with the following remarks.

\begin{remark} \label{concludingremarks}{\em 
(i)   Suppose that $R^\times $ contains infinitely many irreducible elements of $F[t]$.   For example, when $F$ is infinite, let $\mathcal S$ be the linear polynomials $t+\lambda, \lambda\in F$.  When $F=\mathbb F_q$, we take $\mathcal S$ to be the set 
the polynomials of $\{t^{q^r}-t+1\mid r\ge 1\}$.  For any {\it proper} infinite subset $S\subset \mathcal S$, we take $R$ to be 
the localization of $F[t][1/f(t); f(t)\in S]$ so that $F[t]\subset R\subsetneq F(t)$.
Then the group $R^\times/F^\times$ is a free abelian group of infinite rank.   It is easy that the there are $\aleph_1$ many pairwise non-isomorphic 
subgroups $D\subset R^\times$.  
The corresponding collection of groups $\{H(D)\}$ are then pairwise non-isomorphic.  When $n\ge 4$, each of them 
has the $R_\infty$-property. \\
(ii)  Let $F$ be an algebraically closed field of characateristic $p>0$.  Lang \cite{lang} has shown that if $\rho:G\to G$ a Frobenius endomorphism $g\mapsto g^q$ where $q=p^e$, then the Lang map $L:G\to G$ defined as $L(z)
=z^{-1}.\rho(z)$ is surjective.  (Lang's result is more general. See also \cite{steinberg}.) 
The Frobenius endomorphism is an isomorphism of the underlying abstract group $G$, although it is 
not an isomorphism of {\it algebraic varieties}. 
This shows, in 
particular, that $\GL_n(F), \SL_n(F)$ do not have the $R_\infty$-property. In fact, as has been observed already 
by Felshtyn and Nasybullov \cite{fn}, the Lang-Steinberg theorem implies that any semisimple 
connected linear algebraic group $G(F)$ fails to have the $R_\infty$-property.  
}
\end{remark}


\noindent
{\bf Acknowledgments:}  This project was initiated when both the authors were members of The Institute of 
Mathematical Sciences (Homi Bhabha Research Institute), Chennai.  We are grateful to the referee 
of an earlier version of this paper for his/her many valuable comments and suggestions.

\end{document}